\documentclass[reqno,  oneside, 12pt]{amsart}
\usepackage{amsmath, mathrsfs, url, color,amssymb}
\usepackage[margin=2cm, a4paper]{geometry}
\usepackage{enumitem}
\usepackage{bbm}
\usepackage{graphicx}
\usepackage{caption}
\allowdisplaybreaks
\usepackage{color,longtable,graphicx,epstopdf}
\newtheorem{theorem}{Theorem}[section]
\newtheorem{cor}{Corollary}[section]
\newtheorem{lemma}{Lemma}[section]
\newtheorem{proposition}{Proposition}[section]
\newtheorem{remark}{Remark}[section]

\newtheorem{asump}{Hypothesis}[section]
\allowdisplaybreaks 
\DeclareMathOperator{\tr}{trace}

\definecolor{darkgreen}{rgb}{0,.6,0}

\newcommand{\E}{\mathbb E}
\newcommand{\PP}{\mathbb P}

\title[Randomized-tamed Milstein Scheme for SDEs]{A Randomized Milstein Scheme for  SDEs with Superlinear Drift Coefficient}

\usepackage[foot]{amsaddr}

\author{Sani Biswas$^{\mbox{\textasteriskcentered}}$}
\address{\textasteriskcentered$_{\normalfont{\mbox{Corresponding Author}}}$}
\address{\textasteriskcentered$_{\normalfont{\mbox{Centro de Modelamiento Matem\'atico, Universidad de Chile \& IRL 2807 - CNRS. E-Mail: {sbiswas@cmm.uchile.cl}}}}$}

\begin{document}
	
	\begin{abstract}
		This work presents a randomized-tamed Milstein scheme for stochastic differential equations whose drift coefficient exhibits superlinear growth in the state variable and limited temporal regularity, quantified by $\beta$-H\"older continuity with $\beta \in (0,1]$. The scheme combines a taming mechanism to control the superlinear state dependence with a drift randomization strategy designed to address the challenges posed by low temporal regularity. Under suitable assumptions on temporal smoothness, the scheme achieves an optimal strong $\mathscr{L}^p$-convergence rate of order one.

		
	\end{abstract}
	
	\maketitle
\section{Introduction} \label{sec:intro}
Consider a probability space \((\tilde{\Omega}, \tilde{\mathscr{G}}, \tilde{\mathbb{P}})\) equipped with a filtration \(\{\tilde{\mathscr{G}}_t\}_{t \in [0,T]}\) satisfying the usual conditions. Let \( w = \{w_t\}_{t \in [0,T]} \) be a standard \(\mathbb{R}^m\)-valued Brownian motion defined on this filtered probability space.
Suppose the mappings
\[
\mu : [0,T] \times \mathbb{R}^d \to \mathbb{R}^d \quad \text{and} \quad \rho : [0,T] \times \mathbb{R}^d \to \mathbb{R}^{d \times m}
\]
are measurable with respect to the product \(\sigma\)-algebra \(\mathscr{B}([0,T]) \otimes \mathscr{B}(\mathbb{R}^d)\), where \( T > 0 \) is fixed.

The process \(\{x_t\}_{t \in [0,T]}\) satisfies the stochastic differential equation (SDE)
\begin{equation}\label{eq:sde}
	x_t = x_0 + \int_0^t \mu(s, x_s) \, \mathrm{d}s + \int_0^t \rho(s, x_s) \, \mathrm{d}w_s
\end{equation}
almost surely for every \( t \in [0,T] \). The initial condition \( x_0 \) is assumed to be \(\tilde{\mathscr{G}}_0\)-measurable and independent of the Brownian motion \( w \).
The purpose of this work is to introduce and study a randomized--tamed Milstein scheme for stochastic differential equations \eqref{eq:sde}. We consider the setting where the coefficients may grow superlinearly, and where the drift has low time regularity, specifically \(\varrho\)--H\"older continuity in time for some \(\varrho \in (0,1]\).

Stochastic differential equations (SDEs) driven by Brownian motion are widely used in fields such as finance, engineering, biology, and economics to model systems subject to continuous random fluctuations. For instance, they describe the unpredictable yet smooth evolution of systems influenced by numerous small perturbations~\cite{Kloeden1992,Mao2007}.
Although these classical SDEs are powerful for describing diffusion-like dynamics, they do not inherently capture rare, sudden changes observed in many real phenomena. Extensions involving L\'evy noise have been developed to incorporate jumps and discontinuities \cite{Cont2004,Oksendal2007,Situ2006}, but their analysis and numerical treatment are more intricate and fall outside the scope of this paper.

Because explicit solutions to most  SDEs are unavailable, numerical approximation methods play a vital role. Over the years, both explicit and implicit schemes have been proposed, with rigorous studies on their convergence properties. Comprehensive treatments of these numerical approaches can be found in \cite{Dereich2011,Higham2006, Kloeden1992, Platen1999,  Platen2010}.

\subsection*{A Review of Tamed and Randomized Numerical Schemes for SDEs:}
Extensive discussions of both explicit and implicit numerical approaches for SDEs, along with key theoretical foundations, are presented in~\cite{Situ2006}.
A key challenge in this area arises from the failure of classical explicit schemes to control moment growth, often resulting in finite-time blow-up when applied to SDEs with superlinearly growing coefficients, as rigorously demonstrated in~\cite{hutzenthaler2010} for continuous-path SDEs. Although implicit methods offer enhanced stability in such settings, their practical use is frequently constrained by substantial computational overhead, particularly in high-dimensional problems.
To address these issues, recent efforts have concentrated on constructing explicit schemes that balance numerical stability and computational efficiency despite superlinear growth, typically by employing a taming strategy. Significant advances for jump-free SDEs include the works~\cite{hutzenthaler2015, hutzenthaler2020,hutzenthaler2012, Kumar2020, Sabanis2013, Sabanis2016, Tretyakov2013}. Parallel developments concerning L'evy-driven SDEs, although beyond the scope of this work, are detailed in~\cite{Chen2019, Dareiotis2016, Kumar2021a,Kumar2017a, Kumar2017}.

SDEs with time-irregular drift often exhibit convergence issues under the classical Euler scheme due to the irregularity in time, as demonstrated in~\cite[Section~3]{Przybylowicz2014}. 
To address these challenges, modified approaches such as the Euler-integral schemes~\cite{Dareiotis2016, Kumar2017a} and their randomized counterparts~\cite{Przybylowicz2014} have been introduced, leveraging drift integrals or random time evaluations while preserving optimal convergence rates. Extensive error analysis for both jump-free and jump-driven settings is provided in~\cite{Przybylowicz2015a, Przybylowicz2015b, Przybylowicz2022}.
Additional advances in randomized numerical methods include the thorough investigation of Milstein schemes for jump-free SDEs in~\cite{Kruse2019, Pawel2021}. Notably,~\cite{Kruse2019} introduces a two-stage variant of the Milstein scheme and provides a detailed study of its convergence behavior. Extensions to SDEs with jump components are explored in~\cite{Verena2024}, offering new insights into scheme design and error analysis. Furthermore, randomized Milstein schemes tailored to McKean--Vlasov equations with common noise have been developed in~\cite{Biswas2022}, broadening the applicability of these methods to mean-field models.

\subsection*{Obstacles, Innovations and Contributions:}  
	\sloppy
	Theorem~\ref{thm:main_result} establishes the strong $\mathscr{L}^p$-convergence of the proposed randomized-tamed Milstein scheme \eqref{eq:scm} for the SDE~\eqref{eq:sde} with superlinear drift, achieving the optimal convergence rate of order~$1.0$. This result aligns with prior work under suitable H\"older continuity assumptions on the linear drift and diffusion coefficients.
	Recent studies have shown that randomizing the drift coefficient can improve Milstein-type schemes for time-inhomogeneous SDEs; see~\cite{Kruse2019,Pawel2021,Verena2024}. However, randomization alone is insufficient to control the superlinear growth in the drift’s state component. This necessitates incorporating taming strategies alongside randomization by applying the taming approach from~\cite{Kumar2020} directly to the randomized drift.
	The resulting {randomized-tamed Milstein scheme}~\eqref{eq:scm} effectively accommodates drifts with low time regularity and achieves a higher convergence order than the standard tamed Milstein method.


However, the combination of taming and randomization introduces substantial analytical challenges that lie beyond the reach of existing convergence theories. Addressing these complexities requires the development of novel techniques to rigorously establish strong convergence of the scheme \eqref{eq:scm}. Consequently, the analysis introduces an auxiliary process governed by~\eqref{eq:auxiliary_equation}, which plays a key role in connecting the original dynamics~\eqref{eq:sde} with the numerical scheme~\eqref{eq:scm} and in facilitating the derivation of essential estimates. These estimates, which form the backbone of the convergence analysis, are presented in Lemmas~\ref{lem:estimate_first_term_MR} and~\ref{lem:estimate_last_term_MR}.

To the best of our knowledge, this work represents the first rigorous study addressing randomized version of tamed Milstein scheme for SDEs with superlinear coefficients, extending novelty even to deterministic differential equations. The analytical framework developed herein establishes a foundation for future investigations into the strong convergence of higher-order randomized-tamed schemes for L\'evy-driven SDEs with superlinear growth, as well as for McKean–Vlasov-type SDEs.

Complementing this, our framework naturally applies to neuroscience models such as the well-studied FitzHugh--Nagumo system.
Consider the stochastic version of this model from \cite{Baladron2012, Tuckwell2003} with constant noise intensities. The time-inhomogeneous system with multiplicative noise reads
\begin{align}
	\label{eq:sde-system}
	\begin{aligned}
		dV_t &= \bigl( V_t - \tfrac{V_t^3}{3} - R_t + I_{\mathrm{ext}}(t) \bigr) dt + \sigma V_t \, dw_t, \\[0.8ex]
		dR_t &= \alpha ( V_t + \gamma - \lambda R_t ) dt
	\end{aligned}
\end{align}
where \(V_t\) denotes the membrane potential, \(R_t\) the recovery variable, and \(I_{\mathrm{ext}}(t)\) is a time-dependent external input with appropriate regularity. The positive constants \(\sigma, \alpha, \gamma\) and \( \lambda\) govern system parameters, and \(w_t\) is a  standard Brownian motion.
Further, Section~\ref{sec:numerics} presents numerical experiments based on this model to validate the theoretical findings. 

\subsection*{Organization}  Section \ref{sec:randomized_scm} presents the randomized-tamed Milstein scheme together with the statement of the main convergence result. The corresponding proofs are provided in Section \ref{sec:convergence_rate}, while Section \ref{sec:numerics} reports numerical experiments that substantiate the theoretical analysis.

This section concludes with a summary of useful notations.
\subsection*{Notations} 
The symbol \( |\cdot| \) denotes both the Euclidean norm on \( \mathbb{R}^d \) and the Frobenius norm on \( \mathbb{R}^{d \times m} \), depending on the context. For a topological space \( \mathcal{T} \), the corresponding Borel \(\sigma\)-algebra is denoted by \( \mathscr{B}(\mathcal{T}) \).
For a vector \( v \in \mathbb{R}^d \), the \( i \)-th component is denoted by \( v^i \), for \( i \in \{1, \dots, d\} \). If \( M \in \mathbb{R}^{d \times m} \) is a matrix, then \( M^{ij} \) denotes the entry in the \( i \)-th row and \( j \)-th column, for \( i \in \{1, \dots, d\} \) and \( j \in \{1, \dots, m\} \). The transpose and trace of \( M \) are denoted by \( M^* \) and \( \operatorname{trace}(M) \), respectively.
For a function \( f : [0,T] \times \mathbb{R}^d \to \mathbb{R} \), the gradient and Hessian with respect to the state variable \( x \) are denoted by
\(\partial_x f : [0,T] \times \mathbb{R}^d \to \mathbb{R}^d\)  
and \(\partial_x^2 f : [0,T] \times \mathbb{R}^d \to \mathbb{R}^{d \times d}\), respectively.
The expectation operator \( \mathbb{E} \) corresponds to integration with respect to the underlying probability measure. The notation \( X \in \mathscr{L}^p \) indicates that the random variable \( X \) has a finite \( p \)-th moment, i.e.,
\(
\mathbb{E}|X|^p < \infty.
\)
A generic positive constant that does not depend on the discretization parameter is denoted by \( C \); its value may vary between different occurrences.


\section{Preliminaries and Main Result}  
This section begins by addressing the well-posedness and moment stability of the SDE~\eqref{eq:sde}. Subsequently, Theorem~\ref{thm:main_result} states the principal result of this work, namely the strong convergence rate of the proposed randomized-tamed Milstein scheme~\eqref{eq:scm} for the SDE~\eqref{eq:sde}. 

\subsection*{Well-posedness and Moment Bound:}
To guarantee the well-posedness and moment stability of the SDE \eqref{eq:sde}, assume throughout that $q \geq 2$ is fixed.

\begin{asump}\label{asum:ic}
	$\E|x_0|^{q}<\infty$.
\end{asump}
\begin{asump} \label{asum:monotonocity}
	For any $s\in[0,T]$ and $x, y\in \mathbb R^d$, there is a   constant $L>0$,  such that
	\begin{align*}
		(x-y)\big(\mu(s,x)-\mu(s,y)\big)\vee |\rho(s,x)-\rho(s,y)|^2 
		&\leq   L|x-y|^2,
		\\
		|\rho(s,0)|&\leq L.
	\end{align*}
\end{asump}

\begin{asump}
	\label{asum:coercivity}
	For any \( s \in [0,T] \) and \( x \in \mathbb{R}^d \), there is a constant \( L>0 \), such that
	\begin{align*}
		x\mu(s,x) \leq  L(1+|x|^{2}).
	\end{align*}
	
\end{asump}

\begin{asump} \label{asum:continuity}
	For any $s\in  [0,T]$, the map \, $\mathbb R^d\ni x\mapsto \mu(s,x)$ is  continuous.
\end{asump}

The following proposition establishes the well-posedness and moment bounds for the SDE~\eqref{eq:sde}, as a specialization of the general result presented in~\cite{Mao2007}.

\begin{proposition} \label{prop:mb:sde}
	Suppose that Hypotheses~\mbox{\normalfont\ref{asum:monotonocity}} to~\mbox{\normalfont\ref{asum:continuity}} hold. Then the SDE~\eqref{eq:sde} admits a unique strong solution. Moreover, if Assumption~\mbox{\normalfont\ref{asum:ic}} also holds, there exists a constant $C > 0$, such that
	\(
	\sup_{t \in [0,T]} \E|x_t|^q \leq C.
	\)
\end{proposition}

\subsection{Randomized-Tamed Milstein Scheme and Main Result}  \label{sec:randomized_scm}
As already discussed, to construct a randomized version of a Milstein-type scheme 
for time-inhomogeneous SDEs with superlinear drift, it is necessary to combine randomization and taming in the drift coefficient. This leads to the \textit{randomized-tamed Milstein scheme}.

To define the scheme, consider a sequence  $\{(\widehat{\mu})^{n}_{\mathrm{tm}}(s,x) : s \in [0,T],\, x \in \mathbb{R}^d\}_{n \in \mathbb{N}}$ of \(\mathscr{B}([0,T]) \otimes \mathscr{B}(\mathbb{R}^d)\)-measurable functions, satisfying Hypotheses~\textnormal{\ref{asum:tame}}, ~\textnormal{\ref{asum:coercivity_scm}} and  \textnormal{\ref{asum:convergence}}, where the subscript ``$\mathrm{tm}$'' indicates a taming modification ensuring stability and convergence rate under superlinear growth. 
Let \(\tau := (\tau_{j-1})_{j\in\mathbb{N}}\) be an i.i.d.\ sequence of \(\mathrm{Unif}[0,1]\) random variables on \((\Omega^{\tau}, \mathscr{G}^{\tau}, \PP^{\tau})\). For a uniform partition \(0=t_0< t_1<\cdots<t_n=T\) with mesh size \(\delta t := t_j-t_{j-1}\), we set \(\mathscr{G}_0^{\tau}:=\{\emptyset,\Omega^\tau\}\) and \(\mathscr{G}_t^{\tau} := \sigma(\tau_0,\dots,\tau_{j-1})\) for \(t\in(t_{j-1},t_j]\), \(j\in\{1,\ldots,n\}\), and write \(\mathcal{G}^\tau := (\mathscr{G}_t^\tau)_{t\in[0,T]}\). The sequence \((\tau_{j-1})_{j\geq 0}\) is assumed independent of the initial condition \(x_0\) and of the Brownian motion \(w\). Expectation with respect to \(\PP^\tau\) is denoted by \(\E^\tau\). We then define the product probability space \((\Omega,\mathscr{G},\PP) := (\tilde{\Omega}\times\Omega^{\tau},\, \tilde{\mathscr{G}}\otimes \mathscr{G}^{\tau},\, \tilde{\PP}\otimes \PP^{\tau})\) with filtration \((\mathscr{G}_t)_{t\in[0,T]}\), where \(\mathscr{G}_t := \tilde{\mathscr{G}}_t \otimes \mathscr{G}_t^\tau\), augmented to satisfy the usual conditions, and denote expectation with respect to \(\PP\) by \(\E\). Finally, for \(s\in[t_{j-1},t_j)\), we set \(\eta_s^n := t_{j-1}\) and \(\tau_s^n := t_{j-1}+\delta t\,\tau_{j-1}\).

Now, the following {randomized-tamed Milstein scheme} is introduced to approximate the solution of SDE~\eqref{eq:sde}
\begin{align}  
	x_t^{n} &= x_{0} + \int_{0}^{t} (\widehat{\mu})_{\mathrm{tm}}^{n}\big(\tau_s^n, x_{\eta_s^n}^{n}\big) \, \mathrm{d}s + \int_{0}^{t} \Gamma\big(\eta_s^n, s, x_{\eta_s^n}^{n}\big) \, \mathrm{d}w_s \label{eq:scm}
\end{align}
almost surely for any $t \in [0,T]$, where $\Gamma$ is a $d \times m$ matrix, whose $(i,j)$-th element is  defined by
\begin{align}
	\Gamma^{ij}\big(\eta_s^n, s, x_{\eta_s^n}^{n}\big) := \rho^{ij}(\eta_s^n, x_{\eta_s^n}^{n}) + \int_{\eta_s^n}^{s} \partial_x \rho^{ij}\big(\eta_r^n, x_{\eta_r^n}^{n}\big) \rho\big(\eta_r^n, x_{\eta_r^n}^{n}\big) \, \mathrm{d}w_r \label{eq:Gam}
\end{align}
for any $s \in [0,T]$,  $i\in\{1,\ldots,d\}$ and $j\in\{1,\ldots,m\}$.

\begin{remark}
	\sloppy
		While the diffusion coefficient in \cite{Kumar2020} may be superlinear, the randomized-tamed Milstein scheme~\eqref{eq:scm} currently covers only the linear diffusion case, with the schemes in \cite{Kumar2020} recovered as special deterministic instances without time randomization.
		Moreover, when both drift and diffusion exhibit linear growth, scheme~\eqref{eq:scm} reduces to the randomized Milstein method without taming, as studied in~\cite{Verena2024}.
\end{remark}

The convergence rate of the scheme~\eqref{eq:scm} for the SDE~\eqref{eq:sde} is established in Theorem~\ref{thm:main_result} below, subject to the following assumptions. To begin, fix an arbitrary constant $\xi > 0$.

\begin{asump} \label{asum:tame}
	There is
	a constant $L>0$, independent of $n\in\mathbb N$, such that 
	\begin{align*}
		|(\widehat{\mu})_{\mathrm{tm}}^{n}(s,x)|&\leq   L \min\big\{n^{\frac{1}{2}}\big(1+|x|\big),|\mu(s,x)|\big\}
	\end{align*}
	for any  $s\in[0,T]$ and $x\in\mathbb R^d$.
\end{asump}

\begin{asump} \label{asum:coercivity_scm}
	There is a constant \( L>0 \), independent of $n\in \mathbb N$, such that
	\begin{align*}
		& x(\widehat{\mu})_{\mathrm{tm}}^{n}(s,x)	 \leq L\big(1+|x|^2)
	\end{align*}
	for any $s\in[0,T]$ and $x\in\mathbb R^d$.
\end{asump}

\begin{asump} \label{asum:poly_lip_drift}
	For any $s\in[0, T]$ and $x, y\in\mathbb R^d$, there is a constant $L>0$,  such that 
	\begin{align*}
		|\mu(s,x)-\mu(s,y)|
		&\leq L \left(1+|x|+|y|\right)^{\xi}|x-y|,
		\\
		|\mu(s,0)|
		&\leq L.
	\end{align*}
\end{asump}

\begin{asump}  \label{asum:holder_time_diffusion_jump}
	For any  $s,t\in[0,T]$ and  $x\in\mathbb R^d$, there are    constants $L>0$ and $\beta\in (0,1]$,   such that
	\begin{align*}
		|{\mu}(s,x)-{\mu}(t,x)|&\leq L|s-t|^{\beta} \big(1+|x|^{\xi+1}\big),
		\\
		|{\rho}(s,x)-{\rho}(t,x)|&\leq L|s-t|^{\beta+\frac{1}{2}} \big(1+|x|\big).
	\end{align*}
\end{asump} 

\begin{asump} \label{asum:convergence}
	For any  $s\in[0,T]$,  there 
	is
	a constant $L>0$,   independent of $n\in\mathbb N$,    such that 
	\begin{align*}
		&\E\big|\mu\big(\tau^n_s,x_{\eta^n_s}^{n}\big) -(\widehat\mu)_{\mathrm{tm}}^{n}\big(\tau^n_s,x_{\eta^n_s}^{n}\big)\big|^{q}
		\leq L n^{-q}.
	\end{align*}
\end{asump}

\begin{asump} \label{asum:lip_first_derv_rho}
	There exists a   constant $L>0$,
	such that 
	\begin{align*}
		|\partial_x\rho^{ij}(s,x)-\partial_x\rho^{ij}(s,y)|
		&\leq L |x-y|
	\end{align*}
	for any $s\in[0,T]$, $x, y\in \mathbb R^d$, $i\in\{1,\ldots,d\}$ and $j\in\{1,\ldots,m\}$.
\end{asump}

\begin{asump} \label{asum:suplin_second_derv_mu}
	For any $s\in[0,T]$ and $x, y\in \mathbb R^d$, there is   a constant $L>0$,
	such that 
	\begin{align*}
		|\partial_x^2\mu(s,x)|
		&\leq L \big(1+|x|^{\xi-1}\big).
	\end{align*}
\end{asump}

The principal contribution of this study is articulated in the following theorem. A complete and detailed proof is presented in Section~\ref{sec:convergence_rate}.
\begin{theorem} \label{thm:main_result}
	If Hypotheses \mbox{\normalfont  \ref{asum:ic}}
	to \mbox{\normalfont  \ref{asum:suplin_second_derv_mu}} hold,
	then 
	\begin{align*}
		\sup_{ t\in[0,T]}\E|x_t-x_t^{n}|^{p}\leq  C n^{-p\min\{\beta+\frac{1}{2},1\}}
	\end{align*}
	provided that $2p(\xi+2)\leq q$, where  $C>0$ is a constant  independent of $n\in\mathbb N$.
\end{theorem}

\begin{remark}
	Theorem~\mbox{\normalfont  \ref{thm:main_result}} establishes that, for \(\beta \in [0, 1]\), the randomized-tamed Milstein scheme~\eqref{eq:scm} achieves an \(\mathscr{L}^p\)-convergence rate of up to \(\min\{\beta + \tfrac{1}{2}, 1\}\), which surpasses the rate \(\beta\) attained by the classical tamed Milstein scheme.  Furthermore, under appropriate regularity assumptions on the coefficients--specifically, when \(\beta \in [\tfrac{1}{2}, 1]\)-the randomized-tamed Milstein scheme~\eqref{eq:scm} attains the optimal \(\mathscr{L}^p\)-convergence rate of order 1. This theoretical advantage is also confirmed by the numerical experiments presented in Section~\mbox{\normalfont  \ref{sec:numerics}}.	
\end{remark}

Following \cite{Kumar2020}, a concrete form of the taming procedure can be described as follows.

\subsection{An Illustrative Example of Drift Taming}

A specific instance of the general framework in Equation~\eqref{eq:scm} is provided by the randomized-tamed Milstein scheme for the SDE~\eqref{eq:sde}, in which the drift coefficient is modified through the taming
\begin{equation} \label{eq:taming_drift}
	(\widehat{\mu})_{\mathrm{tm}}^{n}(s,x) := \frac{\mu(s,x)}{1 + n^{-1} |x|^{2\xi}}, \quad s \in [0,T], \; x \in \mathbb{R}^d.
\end{equation}
This formulation limits the growth of the drift for large values of $|x|$, thereby ensuring that explicit numerical schemes remain applicable in the presence of superlinear coefficients. The effectiveness of this taming procedure is clearly demonstrated in Section \ref{sec:numerics} through the numerical experiments for the FitzHugh--Nagumo system~\eqref{eq:sde-system}.

\subsection{Moment bound of the randomized-tamed Milstein scheme \eqref{eq:scm}} 
The $\mathscr{L}^q$-moment bound for the randomized-tamed scheme~\eqref{eq:scm} is established in this section (see Lemma~\ref{lem:scm_mb}).

The following remark highlights essential consequence of the preceding assumptions.
These will be instrumental in proving the moment bound and subsequent results.

\begin{remark} \label{rem:super_linear}
	Under Hypotheses~\textnormal{\ref{asum:monotonocity}}, \textnormal{\ref{asum:lip_first_derv_rho}} and \textnormal{\ref{asum:poly_lip_drift}}, there is a constant \( C > 0 \), such that
	\begin{align*}
		|\mu(s,x)| \leq C \big(1 + |x|^{\xi + 1}\big), \quad
		|\partial_x &\mu(s,x)| \leq C \big(1 + |x|^{\xi}\big), \quad
		|\rho(s,x)| \leq C \big(1 + |x|\big), \\
		|\partial_x \rho^{ij}(s,x)| &\leq C, \quad
		|\partial_x^2 \rho^{ij}(s,x)| \leq C
	\end{align*}
	for any \( s\in [0,T] \), \( x \in \mathbb{R}^d \), $i\in\{1,\ldots,d\}$ and $j\in\{1,\ldots,m\}$.
\end{remark}

\begin{lemma} \label{lem:one_step_error}
	If Hypotheses  \mbox{\normalfont  \ref{asum:monotonocity}} and \mbox{\normalfont  \ref{asum:tame}}  are satisfied, then 
	\begin{align*}
		\E|x_s^{n}-x_{\eta^n_s}^{n}|^{q}\leq  C
		n^{-\frac{q}{2}}\big(1+\E|x_{\eta^n_s}^{n}|^{q}\big)
	\end{align*}
	for any ${s\in[0,T]}$, where  $C>0$ is a constant  independent of  $n\in\mathbb N$.
\end{lemma}

\begin{proof} 
	Firstly,  estimate the following term by recalling \eqref{eq:Gam} and utilizing martingale and H\"older's  inequalities together with Remark \ref{rem:super_linear} to obtain 
	\begin{align} \label{eq:estimate_gamma}
		\E\big|  \Gamma^{ij}(\eta^n_s,s,x_{\eta^n_s}^{n})\big|^{q}&\leq C\E\big|  \rho^{ij}(\eta^n_s,x_{\eta^n_s}^{n})\big|^{q}+	C\E\Big| \int_{\eta^n_s}^{s} \partial_x\rho^{ij}(\eta^n_r,x_{\eta^n_r}^{n})\rho(\eta^n_r,x_{\eta^n_r}^{n})\,\mathrm{d}w_r\Big|^{q}  \notag
		\\
		&\leq\, \!C\big(1+\E|x_{\eta^n_s}^{n}|^{q}\big)\!+\!Cn^{-\frac{q}{2}+1}\E \int_{\eta^n_s}^{s}\big| \partial_x\rho^{ij}(\eta^n_r,x_{\eta^n_r}^{n})\rho(\eta^n_r,x_{\eta^n_r}^{n})\big|^{q}\,\mathrm{d}r \notag
		\\
		&\leq\, C\big(1+\E|x_{\eta^n_s}^{n}|^{q}\big)
	\end{align}
	for any $s\in[0,T]$, $i\in\{1,\ldots,d\}$ and $j\in\{1,\ldots,m\}$. Then, recall   \eqref{eq:scm}, and apply H\"older's and martingale  inequalities  to get
	\begin{align*} 
		\E|x_s^{n}-&x_{\eta^n_s}^{n}|^{q}			
		\leq  Cn^{-{q}+1}\E\int_{\eta^n_s}^s \big|(\widehat\mu)_{\mathrm{tm}}^{n}\big(\tau^n_r,x_{\eta^n_r}^{n}\big) \big|^{q}\, \mathrm{d}r 
		\\
		&\qquad\qquad\quad +Cn^{-\frac{q}{2}+1}\E\int_{\eta^n_s}^{s}\big|  \Gamma(\eta^n_r,r,x_{\eta^n_r}^{n})\big|^{q}  \, \mathrm{d}r 
	\end{align*}
	for any $s\in[0,T]$, which on  using Assumption  \ref{asum:tame} and  \eqref{eq:estimate_gamma} completes the proof. 
\end{proof}

In the subsequent lemma, one observes the moment bound of the  scheme \eqref{eq:scm}. 
\begin{lemma} \label{lem:scm_mb}
	If Hypotheses  \mbox{\normalfont  \ref{asum:ic}}, \mbox{\normalfont  \ref{asum:monotonocity}},      \mbox{\normalfont  \ref{asum:tame}} and {\mbox{\normalfont  \ref{asum:coercivity_scm}}} hold, then 
	\begin{align*}
		\sup_{ t\in[0,T]}\E|x_t^{n}|^{q}\leq  C
	\end{align*}
	where $C>0$ is a constant  independent  of $n\in\mathbb N$.
\end{lemma}
\begin{proof}
	Recall the  scheme   \eqref{eq:scm}  and  employ   It\^o's formula (see \cite[Theorem 94]{Situ2006}) to obtain
	\begin{align*} 
		|x_t^{n}|^{q}  
		&\leq
		|x_0|^{q}
		+{q}\int_{0}^{t}|x_s^{n}|^{{q}-2}x_s^{n}\big((\widehat\mu)_{\mathrm{tm}}^{n}\big(\tau^n_s,x_{\eta^n_s}^{n}\big)\big) 
		\,\mathrm{d}s \notag
		\\
		&\quad +{q}\int_{0}^{t}|x_s^{n}|^{{q}-2}x_s^{n}   \,\Gamma(\eta^n_s,s,x_{\eta^n_s}^{n})   \, \mathrm{d}w_s \notag
		\\
		&\quad +\tfrac{q(q-1)}{2}\int_{0}^{t}|x_s^{n}|^{{q}-2}\big|   \Gamma(\eta^n_s,s,x_{\eta^n_s}^{n}) \big|^2 \,\mathrm{d}s \notag
	\end{align*}
	almost surely for any $t\in[ 0, T]$, which on  taking expectation, and utilizing  Young's inequality,  Hypothesis   \ref{asum:coercivity_scm} and the estimate in \eqref{eq:estimate_gamma} yields  
	\begin{align}  
		\E|x_t^{n}|^{q}  &\leq  \E|x_0|^{q}+{q}\E\int_{0}^{t} |x_{s}^{n}|^{{q}-2} x_{\eta^n_s}^{n}(\widehat\mu)_{\mathrm{tm}}^{n}\big(\tau^n_s,x_{\eta^n_s}^{n}\big) \,\mathrm{d}s \notag
		\\
		&\qquad\qquad\qquad+C\E\int_{0}^{t}\Big(|x_{s}^{n}|^{{q}}+\big|  \Gamma(\eta^n_s,s,x_{\eta^n_s}^{n})|^q\Big)\,\mathrm{d}s
		\notag
		\\
		&\quad+ q\E\int_{0}^{t}|x_{s}^{n}|^{{q}-2}\big(x_s^{n}-x_{\eta^n_s}^{n}\big)(\widehat\mu)_{\mathrm{tm}}^{n}\big(\tau^n_s,x_{\eta^n_s}^{n}\big)  \,\mathrm{d}s \notag
		\\
		&  \leq  \E|x_0|^{q}+C+C\int_0^t\sup_{ r\in[0,s]}\E|x_r^{n}|^{q}\,\mathrm{d}s \notag
		\\
		&\quad+ C\E\int_{0}^{t}\big|x_s^{n}-x_{\eta^n_s}^{n}\big|^{{q}/2}\big|(\widehat\mu)_{\mathrm{tm}}^{n}(\tau^n_s,x_{\eta^n_s}^{n})\big|^{{q}/2} \,\mathrm{d}s \notag
	\end{align}
	for any $t\in[ 0, T].$  Further, one uses 	H\"older's inequality along with Hypotheses \ref{asum:ic}, \ref{asum:tame}, and Lemma \ref{lem:one_step_error} to obtain  the following
	\begin{align*}
		\E|x_t^n|^{q}
		&  \leq  \E|x_0|^{q}+C+C\int_0^t\sup_{ r\in[0,s]}\E|x_r^{n}|^{q}\,\mathrm{d}s \notag
		\\
		&\quad+ C\int_{0}^{t}\Big(\E|x_s^{n}-x_{\eta^n_s}^{n}\big|^{{q}}\Big)^{1/2}\Big(\E|(\widehat\mu)_{\mathrm{tm}}^{n}(\tau^n_s,x_{\eta^n_s}^{n})\big|^{{q}}\Big)^{1/2}  \,\mathrm{d}s \notag
		\\
		&  \leq  C+C\int_0^t\sup_{ r\in[0,s]}\E|x_r^{n}|^{q}\,\mathrm{d}s  \notag
	\end{align*}
	which on utilizing  Gr\"onwall's inequality  completes the proof.
\end{proof}

\section{Proof of the Main  Result} \label{sec:convergence_rate}
The proof of the main result stated in Theorem~\ref{thm:main_result} is now presented. As a preliminary step, several intermediate results are established to serve as building blocks in the argument.

The following corollaries are direct consequences of Proposition~\ref{prop:mb:sde} and Lemmas  \ref{lem:one_step_error}, \ref{lem:scm_mb}.
\begin{cor} \label{cor:_coefficients_mb}
	Suppose that Hypotheses~\mbox{\normalfont\ref{asum:ic}}  to \mbox{\normalfont\ref{asum:poly_lip_drift}}and {\mbox{\normalfont\ref{asum:suplin_second_derv_mu}}} hold. Then, for any \( p \geq 2 \) satisfying \( p(\xi + 1) \leq q \), there exists a constant \( C > 0 \), independent of \( n \in \mathbb{N} \), such that
	\begin{align*}
		\sup_{s \in [0,T]} \mathbb{E} \left( |\mu(s, x_s)|^p +   |\rho(s, x_s)|^p+   |\mu(s, x_s^n)|^p +  |\rho(s, x_s^n)|^p\right) &\leq C, 
		\\
		\sup_{s \in [0,T]} \mathbb{E} \left(|\partial_x \mu(s,x_{s})|^{p}+|\partial_x^2 \mu(s,x_{s})|^{p}+|\partial_x \mu(s,x_{s}^n)|^{p}+|\partial_x^2 \mu(s,x_{s}^n)|^{p}\right)&\leq C, 
		\\
		\sup_{s \in [0,T]} \mathbb{E} \left(|(\widehat\mu)_{\mathrm{tm}}^{n}(s, x_s^n)|^p  + |\Gamma(\eta_s^n, s, x_s^n)|^p\right)  &\leq C.
	\end{align*}
\end{cor}

\begin{cor} \label{cor:one_step_error}
	Assume that Hypotheses~{\mbox{\normalfont  \ref{asum:ic}}, \mbox{\normalfont  \ref{asum:monotonocity}},      \mbox{\normalfont  \ref{asum:tame}} and \mbox{\normalfont  \ref{asum:coercivity_scm}}}  are satisfied. Then, for any \( s \in [0, T] \)
	\begin{align*}
		\mathbb{E}|x_s^{n} - x_{\tau^n_s}^{n}|^{p} \leq C n^{-\frac{p}{2}}, \quad \mbox{and } \quad	\mathbb{E}|x_s^{n} - x_{\eta^n_s}^{n}|^{p} \leq C n^{-\frac{p}{2}}
	\end{align*}
	provided that \( p(\xi + 1) \leq q \), where \( C > 0 \) is a constant independent of \( n \in \mathbb{N} \).
\end{cor}

\begin{proof}
	Recall the scheme \eqref{eq:scm} and the definition of $\sigma$-algebra $\mathscr{G}^{\tau}$ from Section \ref{sec:randomized_scm}. Applying H\"older's  and martingale inequalities yields
	\begin{align*}
		\mathbb {E}\left(|x_s^n - x_{\tau^n_s}^n|^p \,\big|\, \mathscr{G}^{\tau} \right)
		&\leq |s - \tau^n_s|^{p - 1} \int_{s}^{\tau^n_s} \mathbb{E}\left( |(\widehat\mu)_{\mathrm{tm}}^{n}({\tau_r^n}, x_{\eta_r^n}^n)|^p \,\big|\, \mathscr{G}^{\tau} \right) \, \mathrm{d}r \\
		&\quad + |s - \tau^n_s|^{\frac{p}{2} - 1} \int_{s}^{\tau^n_s} \mathbb{E}\left( |\Gamma(\eta_r^n,r, x_{\eta_r^n}^n)|^p \,\big|\, \mathscr{G}^{\tau} \right) \, \mathrm{d}r
	\end{align*}
	almost surely for any $s\in[0,T]$, which on taking full expectation and applying Corollary~\ref{cor:_coefficients_mb} yields the result. 
\end{proof}

To derive the rate of convergence of the scheme~\eqref{eq:scm}, the following results are also required.
The proof of the following lemma proceeds by an argument analogous to the one given above.
\begin{cor} \label{cor:one_step_regularity_sde}
	Let Hypotheses~\mbox{\normalfont\ref{asum:ic}} to~\mbox{\normalfont\ref{asum:coercivity}} and \ref{asum:poly_lip_drift} hold. Then, for any \( s,t \in [0, T] \)
	\begin{align*}
		\mathbb{E}|x_s - x_{\tau^n_s}|^p \leq C n^{-\frac{p}{2}},  
		\quad \text{and} \quad 
		\mathbb{E}|x_s - x_t|^p \leq C |s - t|^{\frac{p}{2}}
	\end{align*}
	provided that \( p(\xi + 1) \leq q \), where \( C > 0 \) is a constant independent of \( n \in \mathbb{N} \).
\end{cor}


The  following corollary is a consequence of the above lemma.
\begin{cor} \label{cor:regularity_drift}
	If Hypotheses \mbox{\normalfont\ref{asum:ic}} to~\mbox{\normalfont\ref{asum:coercivity}}, \textnormal{\ref{asum:poly_lip_drift}} and \textnormal{\ref{asum:holder_time_diffusion_jump}} are satisfied, then for any   $s\in[0,T]$ 
	\begin{align*}
		\E|\mu(s,x_s)-\mu(t,x_t)|^{p}&\leq  C|s-t|^{p\min\{\beta,\frac{1}{2}\}}
	\end{align*}
	provided that $ 2 p{(\xi+1)}\leq {q}$,	where $C>0$    is a constant independent of  $n\in \mathbb N$. 
\end{cor}


To establish the convergence rate of the randomized-tamed Milstein scheme~\eqref{eq:scm} for the SDE~\eqref{eq:sde}, the existing analytical techniques for either the randomized Milstein schemes~\cite{Verena2024} or the tamed counterparts~\cite{Kumar2021a, Kumar2017} prove insufficient. This necessitates the development of a new analytical framework. The core challenges are resolved through a two-fold strategy: one component, based on randomization, is addressed in Lemma~\ref{lem:estimate_first_term_MR}, while the other, involving the application of It\^o's lemma together with the taming mechanism, is treated in Lemma~\ref{lem:estimate_last_term_MR}. An essential ingredient of the analysis involves the auxiliary process defined, for all \( t \in [0, T] \), by
\begin{align}
	z_t^n 
	&= x_0 
	+ \int_0^t \mu(\tau^n_s, x_{\tau^n_s})  \,\mathrm{d}s 
	+ \int_0^t \rho(s, x_s) \, \mathrm{d}w_s.
	\label{eq:auxiliary_equation}
\end{align}

\begin{remark} \label{rem:mb_auxilary}
	In view of  Remark~\mbox{\normalfont\ref{rem:super_linear}} and Proposition~\mbox{\normalfont\ref{prop:mb:sde}},  the auxiliary process \( z_t^n \) satisfies   
	\(
	\sup_{t \in [0, T]} \mathbb{E}|z_t^n|^p \leq C,
	\)
	provided that \( p(\xi + 1) \leq q \), where \( C > 0 \) is a constant independent of \( n \in \mathbb{N} \).
\end{remark}
\begin{remark} \label{rem:mb_auxilary_derv}
	As a consequence of Remarks~\mbox{\normalfont\ref{rem:super_linear}}, \mbox{\normalfont\ref{rem:mb_auxilary}}, and Hypothesis~\mbox{\normalfont\ref{asum:suplin_second_derv_mu}}, it holds that
	\[
	\sup_{t \in [0,T]} \mathbb{E}\left(\, |\partial_x \mu(t, z_t^n)|^p + |\partial_x^2 \mu(t, z_t^n)|^p \,\right) \le C
	\]
	provided that \( p(\xi + 1) \leq q \),
	where \(C>0\) is independent of \(n\in\mathbb{N}\).
\end{remark}
The subsequent results provide essential estimates, constituting a pivotal step in establishing the convergence rate of the scheme~\eqref{eq:scm}.
\begin{lemma} \label{lem:estimate_first_term_MR}
	Let Hypotheses \mbox{\normalfont  \ref{asum:ic}} to \mbox{\normalfont \ref{asum:coercivity}}, \textnormal{\ref{asum:poly_lip_drift}} and \textnormal{\ref{asum:holder_time_diffusion_jump}} hold.
	Then, for any   $s\in[0,T]$ 
	\begin{align*}
		&\E|x_s-z_s^{n}|^{p}
		\leq  C  n^{-p\min\{\beta+\frac{1}{2},1\}} 
	\end{align*}
	provided that $p(\xi+1)\leq q$,	where $C>0$    is a constant independent of  $n\in \mathbb N$. 
\end{lemma}
\begin{proof}
	Firstly, introduce the following notations, for any \( t \in [0, T] \),
	\[
	\mathcal{J}(t) := \max\left\{j \in \{0, \ldots, n\} \,\big|\, \frac{jT}{n} \le t \right\},
	\quad 
	\mathcal{K}(t) := \frac{\mathcal{J}(t)T}{n}.
	\]
	Further, recall the definitions of $\tilde{\mathscr G}$ and $\mathbb E^\tau$ from Sections \ref{sec:intro} and \ref{sec:randomized_scm}, respectively.  Then, for each \( j \in \{1, \ldots, n\} \), observe that
	\begin{align*}
		\mathbb{E}\Big(\int_{t_{j-1}}^{t_j}\mu(\tau^n_s, x_{\tau^n_s})\,\mathrm{d}s \,\big|\, \tilde{\mathscr{G}} \Big)
		&= \mathbb{E}^\tau \int_{t_{j-1}}^{t_j} \mu(t_{j-1} + \delta t  \,
		\tau_{j-1},\, x_{t_{j-1} + \delta t \, \tau_{j-1}})  \,\mathrm{d}s \\
		&= \delta t \int_0^1 \mu(t_{j-1} + \delta t \, r,\, x_{t_{j-1} + \delta t \,r}) \,  \mathrm{d}r \\
		&= \int_{t_{j-1}}^{t_j} \mu(s, x_s)\, \mathrm{d}s
	\end{align*}
	which enables the application of \cite[Theorem 4.1]{Kruse2019} to derive the following estimate 
	\begin{align} \label{eq:term_1}
		\mathbb{E}\bigg|\sum_{j=1}^{\mathcal{J}(t)} & \int_{t_{j-1}}^{t_j} \big(\mu(s, x_s) - \mu(\tau^n_s, x_{\tau^n_s})\big) \,\mathrm{d}s \bigg|^p \notag
		\\ 
		&\leq \mathbb{E}\bigg( \mathbb{E}\bigg[ \max_{i \in \{1,\ldots,n\}} \bigg| \sum_{j=1}^{i} \int_{t_{j-1}}^{t_j} \big(\mu(s, x_s) - \mu(\tau^n_s, x_{\tau^n_s})\big) \,\mathrm{d}s \bigg|^p \,\Big|\, \tilde{\mathscr{G}} \bigg] \bigg) \notag \\
		&\leq C \sup_{\substack{s,t \in [0,T] \\ s \neq t}} \bigg( \mathbb{E}|\mu(s, x_s)|^p + \frac{\mathbb{E}|\mu(s, x_s) - \mu(t, x_t)|^p}{|s - t|^{p \min\{\beta, \frac{1}{2}\}}} \,\bigg) n^{-p \min\{\beta + \frac{1}{2}, 1\}} \notag \\
		&\leq C n^{-p \min\{\beta + \frac{1}{2}, 1\}}.
	\end{align}
	for any \( t \in [0,T] \), where for the last step, one utilizes Corollaries \ref{cor:_coefficients_mb} and \ref{cor:regularity_drift}.

	Now, using the definitions of \( x_t \) from \eqref{eq:sde}, \( z_t^n \) from \eqref{eq:auxiliary_equation}, the estimate \eqref{eq:term_1} and H\"older's inequality, one obtains
	\begin{align*} 
		\mathbb{E}|x_t - z_t^n|^p 
		&= \mathbb{E}\left| \int_0^t \big( \mu(s, x_s) - \mu(\tau^n_s, x_{\tau^n_s}) \big) \,\mathrm{d}s \right|^p \\
		&\leq C \mathbb{E}\bigg| \sum_{j=1}^{\mathcal{J}(t)} \int_{t_{j-1}}^{t_j} \big( \mu(s, x_s) - \mu(\tau^n_s, x_{\tau^n_s}) \big) \,\mathrm{d}s \bigg|^p \\
		&\quad + C \mathbb{E}\bigg| \int_{\mathcal{K}(t)}^{t} \big( \mu(s, x_s) - \mu(\tau^n_s, x_{\tau^n_s}) \big) \,\mathrm{d}s \bigg|^p \\
		&\leq C n^{-p \min\{\beta + \frac{1}{2}, 1\}} + C (t - \mathcal{K}(t))^p \sup_{s \in [0,T]} \mathbb{E}\big|\mu(s, x_s)\big|^p
	\end{align*}
	for any \( t \in [0,T] \). Finally, invoking Corollary~\ref{cor:_coefficients_mb} completes the proof.
\end{proof}

\begin{lemma} \label{lem:conv_tame_drift}
	Suppose that Hypotheses \mbox{\normalfont  \ref{asum:ic}} to  \textnormal{\ref{asum:convergence}} are satisfied. Then, for any $s \in [0, T]$
	\begin{align*}
		\E\big|\mu(\tau_s^n,x_{\tau^n_s})-(\widehat {\mu})_{\mathrm{tm}}^{\displaystyle n}(\tau_s^n,x_{\eta^n_s}^{n})\big|^p\leq Cn^{-p\min\{\beta,\frac{1}{2}\}}+ Cn^{\frac{p}{2}}\sup_{ r\in[0,s]}\E|z_r^n-x_r^n|^p
	\end{align*}
	provided that $2p(\xi+1)\leq q$,	where $C>0$    is a constant independent of  $n\in \mathbb N$. 
\end{lemma}

\begin{proof}
	
	Use Hypotheses \ref{asum:poly_lip_drift}, \ref{asum:convergence}, and Young's inequality together with  Proposition \ref{prop:mb:sde}, Lemma \ref{lem:scm_mb}  and Corollaries \ref{cor:one_step_error}, \ref{cor:one_step_regularity_sde} to get 
	\begin{align}
		\E\big|\mu(\tau_s^n&,x_{\tau^n_s})-(\widehat {\mu})_{\mathrm{tm}}^{\displaystyle n}(\tau_s^n,x_{\eta^n_s}^{n})\big|^p
		\leq C	\E\big|\mu(\tau_s^n,x_{s})-{\mu}(\tau_s^n,x_{s}^{n})\big|^p\notag
		\\
		&\quad+C	\E\big|\mu(\tau_s^n,x_{\tau^n_s})-{\mu}(\tau_s^n,x_{s})\big|^p+C	\E\big|\mu(\tau_s^n,x_{s}^n)-{\mu}(\tau_s^n,x_{\eta^n_s}^{n})\big|^p\notag
		\\
		&\quad+C\E\big|\mu(\tau_s^n,x_{\eta^n_s}^n)-(\widehat {\mu})_{\mathrm{tm}}^{\displaystyle n}(\tau_s^n,x_{\eta^n_s}^{n})\big|^p\notag
		\\
		&\leq 	C\E\big((1+|x_{s}|^{p\xi}+|x_{s}^n|^{p \xi})|x_{s}-x_{s}^n|^{p}\big) \notag
		\\
		&\quad+C\E\big((1+|x_{\tau^n_s}|^{p\xi}+|x_{s}|^{p \xi})|x_{\tau^n_s}-x_{s}|^{p}\big) \notag
		\\
		&\quad+C\E\big((1+|x_{s}^n|^{p\xi}+|x_{\eta^n_s}^n|^{p \xi})|x_{s}^n-x_{\eta^n_s}^n|^{p}\big) +Cn^{-p} \notag
		\\
		&\leq 	Cn^{-\frac{p}{2}}\E\big(1+|x_{s}|^{2p\xi+p}+|x_{s}^n|^{2p\xi+p}\big)+Cn^{\frac{p}{2}}\E|x_{s}-x_{s}^n|^{p} \notag
		\\
		&\quad+Cn^{-\frac{p}{2}}\E\big(1+|x_{\tau^n_s}|^{2p\xi}+|x_{s}|^{2p\xi}\big)+Cn^{\frac{p}{2}}\E|x_{\tau^n_s}-x_{s}|^{2p} \notag
		\\
		&\quad+Cn^{-\frac{p}{2}}\E\big(1+|x_{s}^n|^{2p\xi}+|x_{\eta^n_s}^n|^{2p\xi}\big)+Cn^{\frac{p}{2}}\E|x_{s}^n-x_{\eta^n_s}^n|^{2p}  +Cn^{-p} \notag
		\\
		&\leq  Cn^{-\frac{p}{2}}+ Cn^{\frac{p}{2}}\sup_{ r\in[0,s]}\E|x_r-z_r^n|^p + Cn^{\frac{p}{2}}\sup_{ r\in[0,s]}\E|z_r^n-x_r^n|^p \notag
		\\
		&\leq  Cn^{-p\min\{\beta,\frac{1}{2}\}}+ Cn^{\frac{p}{2}}\sup_{ r\in[0,s]}\E|z_r^n-x_r^n|^p  \notag
	\end{align}
	for any $s\in[0,T]$, where the last step follows	from Lemma \ref{lem:estimate_first_term_MR}. {(the proof is changed)}
\end{proof}
\begin{lemma}\label{lem:rho}
	Assume that Hypotheses \mbox{\normalfont  \ref{asum:ic}},  \mbox{\normalfont  \ref{asum:monotonocity}},  \mbox{\normalfont  \ref{asum:tame}} to  \mbox{\normalfont  \ref{asum:poly_lip_drift}}, and \textnormal{\ref{asum:lip_first_derv_rho}} hold. Then, for any $s \in [0, T]$
	\begin{align*}
		\E\big|\rho(\eta^n_s, x_{s}^n)-  \Gamma(\eta^n_s,s,x_{\eta^n_s}^{n})\big|^{p} \ \leq  Cn^{-p}
	\end{align*}
	provided that $2p(\xi+1)\leq q$,	where $C>0$    is a constant independent of  $n\in \mathbb N$. 
\end{lemma}
\begin{proof}
	Firstly, recall \eqref{eq:Gam} and for the notational simplicity, define the following martingale component
	\begin{align}
		\mathfrak M ^{ij}:&=\int_{\eta^n_s}^s \partial_x \rho^{ij}\big(\eta^n_s, x_{r}^{n}\big){\Gamma}\big(\eta^n_r,r,x_{\eta^n_r}^{n}\big) \, \mathrm{d}w_r- \int_{\eta^n_s}^s \partial_x \rho^{ij}\big(\eta^n_r, x_{\eta^n_r}^{n}\big){\rho}\big(\eta^n_r,x_{\eta^n_r}^{n}\big) \, \mathrm{d}w_r  \notag
		\\
		&=\int_{\eta^n_s}^s\big( \partial_x \rho^{ij}(\eta^n_r, x_{r}^{n})-\partial_x \rho^{ij}(\eta^n_r, x_{\eta^n_r}^{n})\big){\rho}(\eta^n_r,x_{\eta^n_r}^{n}) \, \mathrm{d}w_r \notag
		\\
		&\quad+ \int_{\eta^n_s}^s \partial_x \rho^{ij}(\eta^n_r, x_{\eta^n_r}^{n})\int_{\eta^n_r}^r\partial_x{\rho}^{ij}(\eta^n_s,x_{\eta^n_s}^{n}){\rho}(\eta^n_s,x_{\eta^n_s}^{n}) \, \mathrm{d}w_s\,   \mathrm{d}w_r \label{eq:martingale_term}
	\end{align}
	for any $s\in[0,T]$, $i\in\{1,\ldots,d\}$ and $j\in\{1,\ldots,m\}$, whose 
	$p$-th moment  is estimated using martingale   and H\"older's  inequalities, under Hypothesis~\ref{asum:lip_first_derv_rho} and    Remark~\ref{rem:super_linear}, as follows
	\begin{align} \label{eq:estimate_martingale_term}
		\E|\mathfrak M^{ij}|^p&=\,Cn^{-\frac{p}{2}+1}\E\int_{\eta^n_s}^s\big|\partial_x \rho^{ij}(\eta^n_r, x_{r}^{n})-\partial_x \rho^{ij}(\eta^n_r, x_{\eta^n_r}^{n})\big|^p\big|{\rho}(\eta^n_r,x_{\eta^n_r}^{n})\big|^p \, \mathrm{d}r \notag
		\\
		&\quad+
		Cn^{-\frac{p}{2}+1}\E\int_{\eta^n_s}^s\big|\partial_x{\rho}^{ij}(\eta^n_r,x_{\eta^n_r}^{n})\big|^p\Big|\int_{\eta^n_r}^r\partial_x{\rho}^{ij}(\eta^n_s,x_{\eta^n_s}^{n}){\rho}(\eta^n_s,x_{\eta^n_s}^{n}) \, \mathrm{d}w_s\Big|^p \, \mathrm{d}r \notag
		\\
		&\leq Cn^{-\frac{p}{2}+1}\int_{\eta^n_s}^s\big(\E| x_{r}^{n}- x_{\eta^n_r}^{n}|^{2p}\big)^{1/2}\big(\E|{\rho}(\eta^n_r,x_{\eta^n_r}^{n})|^{2p}\big)^{1/2} \, \mathrm{d}r  \notag
		\\
		&\quad+	Cn^{-p+1}\,\E\int_{\eta^n_r}^r\big|\partial_x{\rho}^{ij}(\eta^n_s,x_{\eta^n_s}^{n}){\rho}(\eta^n_s,x_{\eta^n_s}^{n})\big|^p \, \mathrm{d}r \leq
		Cn^{-p}
	\end{align}
	for any $s\in[0,T]$, $p\geq 2$, $i\in\{1,\ldots,d\}$ and $j\in\{1,\ldots,m\}$, where the final inequality follows from  Corollaries   \ref{cor:_coefficients_mb}, \ref{cor:one_step_error} and Remark   \ref{rem:super_linear}. Moreover, recall the scheme \eqref{eq:scm}, and apply It\^o’s formula \cite[Theorem 94]{Situ2006} to the map \( x \mapsto \rho^{ij}(\eta^n_s, x) \) to obtain
	\begin{align*}
		\rho^{ij}(\eta^n_s, x_s^n) - &\Gamma^{ij}(\eta^n_s, s, x_{\eta^n_s}^n)
		= \int_{\eta^n_s}^s \partial_x \rho^{ij}(\eta^n_s, x_r^n)(\widehat \mu)_{\mathrm{tm}}^{n}(\tau^n_r, x_{\eta^n_r}^n) \,  \mathrm{d}r \\
		&\quad + \frac{1}{2} \int_{\eta^n_s}^{s} \tr\Big[\partial_x^2 \rho^{ij}(\eta^n_s, x_r^n) 
		\Gamma(\eta^n_r, r, x_{\eta^n_r}^n)\Gamma^*(\eta^n_r, r, x_{\eta^n_r}^n)\Big] \,  \mathrm{d}r  + \mathfrak{M}^{ij}
	\end{align*}
	for any \( s \in [0,T] \), $i\in\{1,\ldots,d\}$ and $j\in\{1,\ldots,m\}$, where \( \mathfrak{M}^{ij} \) denotes the martingale term defined in \eqref{eq:martingale_term}. 
	To estimate its \( p \)-th moment,  apply H\"older's inequality together with Remark~\ref{rem:super_linear}, yielding
	\begin{align*}
		\E \big|\rho^{ij}(\eta^n_s, x_s^n) &-\Gamma^{ij}(\eta^n_s, s, x_{\eta^n_s}^n) \big|^p \\
		&\leq C n^{-p} \sup_{t \in [0, T]}  \E\big( 
		\big|(\widehat \mu)_{\mathrm{tm}}^{n}(t, x_t^n)\big|^p + 
		\big|\Gamma(\eta_n(t), t, x_t^n)\big|^{2p}
		\big) + \E\big(|\mathfrak{M}^{ij}|^p\big)
	\end{align*}
	for any \( s \in [0,T] \), $i\in\{1,\ldots,d\}$ and $j\in\{1,\ldots,m\}$.
	Finally, invoking Corollary~\ref{cor:_coefficients_mb} and the martingale estimate \eqref{eq:estimate_martingale_term} completes the proof.
\end{proof}

\begin{lemma} \label{lem:mu}
	If  Hypotheses \mbox{\normalfont  \ref{asum:ic}}
	to \mbox{\normalfont  \ref{asum:suplin_second_derv_mu}}  are satisfied,
	then   for any $s \in [0, T]$ and $p\geq 2$ 
	\begin{align*}
		\E\Big(|z_s^n&-x_s^n|^{p-2}(z_s^n-x_s^n)\big(\mu(\tau^n_s,\,x_{\tau^n_s})-\mu(\tau^n_s,z_{s}^n)+ 	\mu(\tau^n_s,x_s^{n})-\mu(\tau^n_s,x_{\eta^n_s}^{n})\big)\Big)
		\\
		& \qquad\qquad \leq   Cn^{-p\min\{\beta+\frac{1}{2},1\}} + C\sup_{ r\in[0,s]}\E|z_r^n-x_r^n|^p
	\end{align*}
	provided that $2p(\xi+2)\leq q$, where  $C>0$ is a constant  independent of $n\in\mathbb N$.
\end{lemma}
\begin{proof}
	Firstly, 	recall Equations \eqref{eq:sde}, \eqref{eq:auxiliary_equation} and  \eqref{eq:scm} to write almost surely for any $s\in[0,T]$
	\begin{align*}
		x_{s}&= x_{\eta_s^n}+\int_{\eta_s^n}^{s}  \mu\big(r,x_{r}\big)\, \mathrm{d}r +\int_{\eta_s^n}^{s}   \rho\big(r,x_{r}\big)\, \mathrm{d}w_r,
		\\
		z_{s}^{n}&= z_{\eta_s^n}^{n}+\int_{\eta_s^n}^{s}  \mu\big(\tau^n_r,x_{\tau^n_r}\big)\, \mathrm{d}r +\int_{\eta_s^n}^{s}   \rho\big(r,x_{r}\big)\, \mathrm{d}w_r, 
		\\
		x_s^{n}&= x_{\eta_s^n}^n+\int_{\eta_s^n}^{s}  (\widehat\mu)_{\mathrm{tm}}^{n}\big(\tau^n_r,x_{\tau^n_r}^{n}\big)\, \mathrm{d}r +\int_{\eta_s^n}^{s} \Gamma\big(\eta^n_r,r,x_{\eta^n_r}^{n}\big)\, \mathrm{d}w_r.    
	\end{align*}
	Then, use  It\^o's  lemma \cite[Theorem 94]{Situ2006} to the map \( x \mapsto \mu^{i}(\tau^n_s, x) \) to obtain 
	\begin{align} 
		\mu^{i}(\tau^n_s,x_{\tau^n_s})-\mu^{i}(\tau^n_s,z_{s}^n)&+ \big(\mu^{i}(\tau^n_s,x_s^{n})-\mu^{i}(\tau^n_s,x_{\eta^n_s}^{n})\big) \notag
		\\
		&=\mu^{i}(\tau^n_s,x_{\eta^n_s})-\mu^{i}(\tau^n_s,z_{\eta^n_s}^n)+\mathcal A^i+\mathcal M^i \label{eq:A_M}
	\end{align}
	for any \( s \in [0,T] \) and $i\in\{1,\ldots,d\}$,	where 
	\begin{align} 
		\mathcal A^i:&\!= \!\int_{\eta^n_s}^{\tau^n_s}\!\!\partial_x \mu^{i}\big(\tau^n_s,x_{r}\big)\mu\big(r,x_{r}\big) \, \mathrm{d}r+ \frac{1}{2}\int_{\eta^n_s}^{\tau^n_s} \!\tr\big[\partial_{x}^2 \mu^{i}\big(\tau^n_s,x_{r}\big) 
		{\rho}\big(r,x_{r}\big){\rho}^*\big(r,x_{r}\big)\big]  \, \mathrm{d}r  \notag
		\\
		&\qquad -\int_{\eta^n_s}^{s}\partial_x \mu^{i}\big(\tau^n_s,z_{r}^n\big)\mu\big(\tau^n_r,x_{\tau^n_r}\big) \, \mathrm{d}r  \notag
		\\
		&\qquad - \frac{1}{2}\int_{\eta^n_s}^{s} \tr\big[\partial_{x}^2 \mu^{i}\big(\tau^n_s,z_{r}^n\big) 
		{\rho}\big(r,x_{r}\big){\rho}^*\big(r,x_{r}\big)\big]  \, \mathrm{d}r \notag
		\\
		&\qquad+\int_{\eta^n_s}^s\partial_x \mu^{i}\big(\tau^n_s,x_{r}^{n}\big)(\widehat \mu)_{\mathrm{tm}}^{n}\big(\tau^n_r,x_{\eta^n_r}^{n}\big) \, \mathrm{d}r  \notag
		\\
		&\qquad+ \frac{1}{2}\int_{\eta^n_s}^{s} \tr\big[\partial_{x}^2 \mu^{i}\big(\tau^n_s,x_r^{n}\big) 
		{\Gamma}\big(\eta^n_r,r,x_{\eta^n_r}^{n}\big){\Gamma}^*\big(\eta^n_r,r,x_{\eta^n_r}^{n}\big)\big]  \, \mathrm{d}r  \notag
		\\
		\mathcal M^i:&= \int_{\eta^n_s}^{\tau^n_s} \partial_x \mu^{i}\big(\tau^n_s, x_{r}\big){\rho}\big(r,x_{r}\big) \, \mathrm{d}w_r-\int_{\eta^n_s}^{s} \partial_x \mu^{i}\big(\tau^n_s, z_{r}^n\big){\rho}\big(r,x_{r}\big) \, \mathrm{d}w_r
		\notag 
		\\
		&\qquad +\int_{\eta^n_s}^s \partial_x \mu^{i}\big(\tau^n_s, x_{r}^{n}\big){\Gamma}\big(\eta^n_r,r,x_{\eta^n_r}^{n}\big) \, \mathrm{d}w_r. \label{eq:M}
	\end{align}
	Further, denote by \( \mathcal{A}:= \big( \mathcal{A}^i \big)_{i \in \{1, \ldots, d\}} \), and   similarly define \( \mathcal{M} := \big( \mathcal{M}^i \big)_{i \in \{1, \ldots, d\}} \).
	Moreover, 	by  Young's inequality together with   Corollary \ref{cor:_coefficients_mb} and Remark \ref{rem:mb_auxilary_derv}
	\begin{align} \label{eq:estimate_A}
		\E|\mathcal A|^p&\leq C\sum_{i=1}^{d}\E|\mathcal A^i|^p \notag
		\\
		&\leq Cn^{-p}\sup_{t\in[0,T]} \sum_{i=1}^{d}\E\Big(|\partial_x \mu^{i}(t,x_{t})\mu(t,x_{t})|^{p}+|\partial_x^2 \mu^{i}(t,x_{t})|^p|\rho(t,x_{t})|^{2p} \notag
		\\
		&\quad+|\partial_x \mu^{i}(t,z_{t}^n)\mu(t,x_{t})|^{p}+|\partial_x^2 \mu^{i}(t,z_{t}^n)|^{p}|\rho(t,x_{t})|^{2p} \notag
		\\
		&\quad+|\partial_x \mu^{i}(t,x_{t}^n)(\widehat \mu)_{\mathrm{tm}}^{n}(t,x_{t}^{n})|^{p}+|\partial_x^2 \mu^{i}(t,x_{t}^n)|^p|{\Gamma}(\eta^n_t,t,x_{t}^{n})|^{2p}\Big) \notag
		\\
		&\leq  Cn^{-p}\sup_{t\in[0,T]} \sum_{i=1}^{d}\E\Big(|\partial_x \mu^{i}(t,x_{t})|^{2p}+|\mu(t,x_{t})|^{2p}+|\partial_x^2 \mu^{i}(t,x_{t})|^{2p} \notag
		\\
		&\quad+|\rho(t,x_{t})|^{4p}+|\partial_x \mu^{i}(t,z_{t}^n)|^{2p}+|\partial_x^2 \mu^{i}(t,z_{t}^n)|^{2p}+|\partial_x \mu^{i}(t,x_{t}^n)|^{2p}  \notag
		\\
		&\quad+|(\widehat \mu)_{\mathrm{tm}}^{n}(t,x_{t}^{n})|^{2p}+|\partial_x^2 \mu^{i}(t,x_{t}^n)|^{2p}+|{\Gamma}(\eta^n_t,t,x_{t}^{n})|^{4p}\Big)\leq Cn^{-p}.
	\end{align}
	Further, by  Martingale, H\"older's and Young's inequalities along with   Corollary \ref{cor:_coefficients_mb} and Remark \ref{rem:mb_auxilary_derv}, one obtains
	\begin{align} \label{eq:estimate_M}
		\E\big|\mathcal M\big|^p&\leq   C\E\Big(\int_{\eta^n_s}^{s\vee\tau^n_s} \big(\,\big|\partial_x \mu(\tau^n_s, x_{r}){\rho}(r,x_{r})\big|^2+ \big|\partial_x \mu(\tau^n_s, z_{r}^n){\rho}(r,x_{r})\big|^2  \notag
		\\
		&\qquad\qquad+ \big| \partial_x \mu(\tau^n_s, x_{r}^{n}){\Gamma}(\eta^n_r,r,x_{\eta^n_r}^{n})\big|^2\big )  \, \mathrm{d}r\Big)^{p/2} \notag
		\\
		&\leq  Cn^{-\frac{p}{2}}\sup_{t\in[0,T]} \E\Big(|\partial_x \mu(t, x_{t})|^{2p}+|{\rho}(t,x_{t})|^{2p}+|\partial_x \mu(t, z_{t}^n)|^{2p}  \notag
		\\
		&\qquad\qquad+ |\partial_x \mu(t, x_{t}^{n})|^{2p}+|{\Gamma}(\eta^n_t,t,x_{t}^{n})|^{2p}\Big)
		\leq Cn^{-\frac{p}{2}} 
	\end{align}
	for any $s\in[0,T]$.
	Furthermore, recall the scheme \eqref{eq:scm} and the auxiliary equation \eqref{eq:auxiliary_equation} to write
	\begin{align*}
		z_s^n-x_s^n&=z_{\eta^n_s}^n-x_{\eta^n_s}^n
		+\int_{\eta^n_s}^s\big(\mu(\tau_r^n,x_{\tau^n_r})-(\widehat {\mu})_{\mathrm{tm}}^{\displaystyle n}(\tau_r^n,x_{\eta^n_r}^{n})\big)\, \mathrm{d}r\notag
		\\
		&
		\quad+\int_{\eta^n_s}^s\big(\rho(r,x_r)-\rho(r,x_r^{n})\big)\, \mathrm{d}w_r
		+\int_{\eta^n_s}^s\big(\rho(r,x_r^{n})-{\Gamma}(\eta^n_r,r,x_{\eta^n_r}^{n})\big)\, \mathrm{d}w_r \notag
	\end{align*}
	for any $s\in[0,T]$.
	Then, apply 
	Young's, H\"older's  and martingale inequalities together with the estimate \eqref{eq:estimate_M}  to obtain the following
	\begin{align} \label{eq:martingale_related_term}
		\E  \Big(\big|(z_s^n-x_s^n)-&(z_{\eta^n_s}^n-x_{\eta^n_s}^n)\big|^{p/2}|\mathcal M|^{p/2}\Big) \notag
		\\
		&
		\leq C\E\Big(\Big|\int_{\eta^n_s}^s\big(\mu(\tau_r^n,x_{\tau^n_r})-(\widehat {\mu})_{\mathrm{tm}}^{\displaystyle n}(\tau_r^n,x_{\eta^n_r}^{n})\big)\, \mathrm{d}r\Big|^{p/2}|\mathcal M|^{p/2}\Big)\notag
		\\
		&\quad
		+C\E\Big(\Big|\int_{\eta^n_s}^s\big(\rho(r,x_r)-\rho(r,x_r^{n})\big)\, \mathrm{d}w_r\Big|^{p/2} |\mathcal M|^{p/2}\Big)  \notag
		\\
		&\quad+C\E\Big(\Big|\int_{\eta^n_s}^s\big(\rho(r,x_r^{n})-{\Gamma}(\eta^n_r,r,x_{\eta^n_r}^{n})\big)\, \mathrm{d}w_r\Big|^{p/2}|\mathcal M|^{p/2}\Big) \notag
		\\
		&\leq   Cn^{\frac{p}{2}}\E\Big|\int_{\eta^n_s}^s\mu(\tau_r^n,x_{\tau^n_r})-(\widehat {\mu})_{\mathrm{tm}}^{\displaystyle n}(\tau_r^n,x_{\eta^n_r}^{n})\, \mathrm{d}r\Big|^p+ Cn^{-\frac{p}{2}}\E|\mathcal M|^p \notag
		\\
		&
		\quad+C\Big(\E\Big|\int_{\eta^n_s}^s\big(\rho(r,x_r)-\rho(r,x_r^{n})\big)\, \mathrm{d}w_r\Big|^{p}\Big)^{1/2}(\E|\mathcal M|^{p})^{1/2} 
		\notag
		\\
		&
		\quad+C\Big(\E\Big|\int_{\eta^n_s}^s\big(\rho(r,x_r^{n})-{\Gamma}(\eta^n_r,r,x_{\eta^n_r}^{n})\big)\, \mathrm{d}w_r\Big|^p\Big)^{1/2}(\E|\mathcal M|^p)^{1/2} 
		\notag
		\\
		&\leq Cn^{-\frac{p}{2}+1}\E\int_{\eta^n_s}^s\big|\mu(\tau_r^n,x_{\tau^n_r})-(\widehat {\mu})_{\mathrm{tm}}^{\displaystyle n}(\tau_r^n,x_{\eta^n_r}^{n})\big|^p \, \mathrm{d}r +Cn^{-p} \notag
		\\
		&
		\quad+Cn^{-\frac{p}{4}}\Big(n^{-\frac{p}{2}+1}\E\int_{\eta^n_s}^s\big|\rho(r,x_r)-\rho(r,x_r^{n})\big|^p\, \mathrm{d}r\Big)^{1/2} \notag
		\\
		&
		\quad+Cn^{-\frac{p}{4}}\Big(n^{-\frac{p}{2}+1}\E\int_{\eta^n_s}^s\big|\rho(r,x_r^n)-\rho(\eta^n_r,x_r^{n})\big|^p\, \mathrm{d}r\Big)^{1/2}  \notag
		\\
		&
		\quad	+Cn^{-\frac{p}{4}}\Big(n^{-\frac{p}{2}+1}\E\int_{\eta^n_s}^s\big|\rho(\eta^n_r,x_r^{n})-{\Gamma}(\eta^n_r,r,x_{\eta^n_r}^{n})\big|^p\, \mathrm{d}r\Big)^{1/2} 
	\end{align}
	for any $t\in[0,T]$. Further, by using Lemmas  \ref{lem:conv_tame_drift}, \ref{lem:rho} and  {Hypotheses \ref{asum:monotonocity}, \ref{asum:holder_time_diffusion_jump}} along with Young's inequality and Lemma \ref{lem:estimate_first_term_MR},  one obtains, 	for any $s\in[0,T]$ 
	\begin{align*}
		\E  \Big(\big|(z_s^n&-x_s^n)-(z_{\eta^n_s}^n-x_{\eta^n_s}^n)\big|^{p/2}|\mathcal M|^{p/2}\Big)
		\\
		&\leq Cn^{-\frac{p}{2}}n^{-p\min\{\beta,\frac{1}{2}\}}+ Cn^{-\frac{p}{2}}n^{\frac{p}{2}}\sup_{ r\in[0,s]}\E|z_r^n-x_r^n|^p+ Cn^{-p} \notag
		\\
		&\quad+Cn^{-\frac{p}{2}}\sup_{ r\in[0,s]}\bigl(\E|x_r-x_r^n|^p\bigr)^{1/2}             \notag
		\\
		&\leq  Cn^{-p\min\{\beta+\frac{1}{2},1\}}+ C\sup_{ r\in[0,s]}\E|z_r^n-x_r^n|^p+Cn^{-{p}}+ C\sup_{ r\in[0,s]}\E|x_r-x_r^n|^p  \notag
		\\
		&\leq  Cn^{-p\min\{\beta+\frac{1}{2},1\}}+ C\sup_{ r\in[0,s]}\E|x_r-z_r^n|^p + C\sup_{ r\in[0,s]}\E|z_r^n-x_r^n|^p \notag
		\\
		&\leq  Cn^{-p\min\{\beta+\frac{1}{2},1\}}+ C\sup_{ r\in[0,s]}\E|z_r^n-x_r^n|^p. 
	\end{align*}
	For further use, define \( f(z) = |z|^{p-2} z \) for \( z \in \mathbb{R}^d \) and \( p \geq 2 \). Then, for all \( x, y \in \mathbb{R}^d \) and any \( \theta \in (0,1) \), there exists a constant \( C = C(p) > 0 \) such that
	\[
	| f(x) - f(y) | 
	\leq C (p - 1) \left| \theta x + (1 - \theta) y \right|^{p-2} | x - y |
	\leq C \left( |x|^{p-2} + |y|^{p-2} \right)| x - y |.
	\]
	One utilizes the above identity along with Young's inequality  and recall $\mathcal M$ from  \eqref{eq:M}  to proceed as follows
	\begin{align}
		\E\Big(\, &|z_s^n - x_s^n|^{p-2}(z_s^n - x_s^n)\mathcal M\,\Big)
		\notag
		\\
		&=\E\Big(\, |z_s^n - x_s^n|^{p-2}(z_s^n - x_s^n)\mathcal M - |z_{\eta_s^n}^n - x_{\eta_s^n}^n|^{p-2}(z_{\eta_s^n}^n - x_{\eta_s^n}^n)\mathcal M \,\Big) \notag
		\\
		&\qquad\qquad\qquad+\E\Big(\,|z_{\eta_s^n}^n - x_{\eta_s^n}^n|^{p-2}(z_{\eta_s^n}^n - x_{\eta_s^n}^n)\mathcal M \,\Big) \notag
		\\
		&\leq \E\Big|\, |z_s^n - x_s^n|^{p-2}(z_s^n - x_s^n)\mathcal M - |z_{\eta_s^n}^n - x_{\eta_s^n}^n|^{p-2}(z_{\eta_s^n}^n - x_{\eta_s^n}^n)\mathcal M \,\Big| \notag
		\\
		&\qquad\qquad\qquad+\E\Big(\,|z_{\eta_s^n}^n - x_{\eta_s^n}^n|^{p-2}(z_{\eta_s^n}^n - x_{\eta_s^n}^n)\E\big(\mathcal M|\mathcal G_{\eta^n_s}\big) \,\Big) \notag
		\\
		&\leq C\E\Big(\big( |z_s^n - x_s^n|^{p-2} + |z_{\eta_s^n}^n - x_{\eta_s^n}^n|^{p-2} \big)
		\big| (z_s^n - x_s^n)- (z_{\eta_s^n}^n - x_{\eta_s^n}^n) \big||\mathcal M|\Big)\notag
		\\
		&\leq C\E\big( |z_s^n - x_s^n|^{p} + |z_{\eta_s^n}^n - x_{\eta_s^n}^n|^{p} \big)+
		C\E  \Big(\big|(z_s^n-x_s^n)-(z_{\eta^n_s}^n-x_{\eta^n_s}^n)\big|^{p/2}|\mathcal M|^{p/2}\Big)\notag
		\\
		&\leq    Cn^{-p\min\{\beta+\frac{1}{2},1\}}+ C\sup_{ r\in[0,s]}\E|z_r^n-x_r^n|^p \label{eq:final_martingale_estimate}
	\end{align}
	for any $s\in[0,T]$, where the last step follows the estimate driven in \eqref{eq:martingale_related_term}.
	
	Finally, to estimate the main inequality of this lemma, recall \eqref{eq:A_M} and apply Young's and H\"older's inequalities together with Hypothesis \ref{asum:poly_lip_drift} and Equations \eqref{eq:estimate_A}, \eqref{eq:final_martingale_estimate} as follows
	\begin{align*}
		\E\Big(|z_s^n&-x_s^n|^{p-2}(z_s^n-x_s^n)\big(\mu(\tau^n_s,\,x_{\tau^n_s})-\mu(\tau^n_s,z_{s}^n)+ 	\mu(\tau^n_s,x_s^{n})-\mu(\tau^n_s,x_{\eta^n_s}^{n})\big)\Big)
		\\
		&=\,\E\Big(|z_s^n-x_s^n|^{p-2}(z_s^n-x_s^n)\Big[\mu(\tau^n_s,x_{\eta^n_s})-\mu(\tau^n_s,z_{\eta^n_s}^n)+\mathcal A+\mathcal M\Big]\Big)
		\\
		&\leq  C\E|z_s^n-x_s^n|^p+C\E|\mu(\tau^n_s,x_{\eta^n_s})-\mu(\tau^n_s,z_{\eta^n_s}^n)|^p
		\\
		&\qquad+C\E|\mathcal A|^p+\E\big(|z_s^n-x_s^n|^{p-2}(z_s^n-x_s^n)\mathcal M\big)
		\\
		&\leq     C\sup_{ r\in[0,s]}\E|z_r^n-x_r^n|^p+C\E\big(1+|x_{\eta^n_s}|^{p\xi}+|z_{\tau^n_s}^n|^{p \xi}\big)|x_{\eta^n_s}-z_{\eta^n_s}^n|^{p}
		\\
		&\qquad+Cn^{-p\min\{\beta+\frac{1}{2},1\}}
		\\
		&\leq    Cn^{-p\min\{\beta+\frac{1}{2},1\}}+ C\sup_{ r\in[0,s]}\E|z_r^n-x_r^n|^p
		\\
		&\qquad+C\Big(\E\big(1+|x_{\eta^n_s}|^{2p\xi}+|z_{\tau^n_s}^n|^{2p \xi}\big)\Big)^{1/2}\Big(\E|x_{\eta^n_s}-z_{\eta^n_s}^n|^{2p}\Big)^{1/2}
		\\
		&\leq    Cn^{-p\min\{\beta+\frac{1}{2},1\}}+ C\sup_{ r\in[0,s]}\E|z_r^n-x_r^n|^p
	\end{align*}
	for any $s\in[0,T]$, where the last step follows from Lemmas \ref{lem:scm_mb} and \ref{lem:estimate_first_term_MR}
\end{proof}

The following lemma is a key ingredient in deriving the main result.
\begin{lemma} \label{lem:estimate_last_term_MR}
	Let	Hypotheses \mbox{\normalfont  \ref{asum:ic}}
	to \mbox{\normalfont  \ref{asum:suplin_second_derv_mu}}  hold.
	Then,  
	\begin{align*}
		&\sup_{ t\in[0,T]}\E|z_t^n-x_t^{n}|^{p}
		\leq  Cn^{-p\min\{\beta+\frac{1}{2},1\}}
	\end{align*}
	provided that $2p(\xi+2)\leq q$, where  $C>0$ is a constant  independent of $n\in\mathbb N$.
\end{lemma}

\begin{proof}
	Firstly, recall the scheme and the  auxiliary system from \eqref{eq:scm} and  \eqref{eq:auxiliary_equation}, respectively. Then, for the case \(p\geq 2\), utilize  It\^o's  lemma \cite[Theorem 94]{Situ2006} to the map \( x \mapsto |x|^p \) to obtain the following
	\begin{align} 
		|z_t^n-x_t^{n}|^{p}&\leq {p}\int_{0}^{t}|z_s^n-x_s^{n}|^{{p}-2}(z_s^n-x_s^{n})\big(\mu(\tau^n_s,x_{\tau^n_s})-(\widehat\mu)_{\mathrm{tm}}^{n}\big(\tau^n_s,x_{\eta^n_s}^{n}\big)\big)\,\mathrm{d}s  \notag
		\\
		&\qquad+p\int_{0}^{t}|z_s^n-x_s^{n}|^{{p}-2}(z_s^n-x_s^{n})\big(\rho(s, x_s)-  \Gamma(\eta^n_s,s,x_{\eta^n_s}^{n})\big)\,\mathrm{d}w_s  \notag
		\\
		& \qquad+\tfrac{p(p-1)}{2}\int_{0}^{t}|z_s^n-x_s^{n}|^{{p}-2}\big|\rho(s, x_s)-  \Gamma(\eta^n_s,s,x_{\eta^n_s}^{n})\big|^2 \,\mathrm{d}s \notag 
	\end{align}
	almost surely for any $t\in[0,T]$, which on taking expectations of both sides and invoking Young's inequality yields 
	\begin{align} \label{eq:error_y_n_x_n}
		\E|z_t^n-x_t^{n}|^{p}
		&\leq 
		p\E\int_{0}^{t}|z_s^n-x_s^{n}|^{{p}-2}(z_s^n-x_s^{n})\big(\mu(\tau^n_s, x_{\tau^n_s})-(\widehat\mu)_{\mathrm{tm}}^{n}\big(\tau^n_s,x_{\eta^n_s}^{n}\big)\big)\,\mathrm{d}s \notag
		\\
		&
		\quad+C\E\int_{0}^{t}|z_s^n-x_s^{n}|^{{p}}\,\mathrm{d}s 
		+C\E\int_{0}^{t}\big|\rho(s, x_s)-  \Gamma(\eta^n_s,s,x_{\eta^n_s}^{n})\big|^p\,\mathrm{d}s \notag
		\\
		&=:  \, \mathscr{T}_\mu +\mathscr{T}_\rho.
	\end{align}
	for any $t\in[0,T]$. 
	To estimate the first term $\mathscr{T}_\mu$ in \eqref{eq:error_y_n_x_n}, one makes use of Lemma \ref{lem:mu}, Hypotheses~\ref{asum:monotonocity} and~\ref{asum:convergence}, together with Young's inequality as follows
	\begin{align}
		\mathscr{T}_\mu: &= p\E\int_{0}^{t}|z_s^n-x_s^{n}|^{{p}-2}(z_s^n-x_s^{n})\big(\mu(\tau^n_s, x_{\tau^n_s})-(\widehat\mu)_{\mathrm{tm}}^{n}\big(\tau^n_s,x_{\eta^n_s}^{n}\big)\big)\,\mathrm{d}s \notag
		\\
		&\leq  p\E\int_{0}^{t}|z_s^n-x_s^{n}|^{{p}-2}(z_s^n-x_s^{n})\big(\mu(\tau^n_s,x_{\tau^n_s})-\mu( \tau^n_s, z_{s}^n)\big)\,\mathrm{d}s \notag
		\\
		&\quad+ p\E\int_{0}^{t}|z_s^n-x_s^{n}|^{{p}-2}(z_s^n-x_s^{n})\big(\mu( \tau^n_s, z_{s}^n)-\mu(\tau^n_s,  x_{s}^n)\big)\,\mathrm{d}s \notag
		\\
		&\quad+p\E\int_{0}^{t}|z_s^n-x_s^{n}|^{{p}-2}(z_s^n-x_s^{n})\big(\mu(\tau^n_s, x_{s}^n)-\mu(\tau^n_s, x_{\eta^n_s}^n)\big)\,\mathrm{d}s \notag
		\\
		&\quad+p\E\int_{0}^{t}|z_s^n-x_s^{n}|^{{p}-2}(z_s^n-x_s^{n})\big(\mu(\tau^n_s, x_{\eta^n_s}^n)-(\widehat\mu)_{\mathrm{tm}}^{n}\big(\tau^n_s,x_{\eta^n_s}^{n}\big)\big)\,\mathrm{d}s \notag
		\\
		&\leq Cn^{-p\min\{\beta+\frac{1}{2},1\}}  +C  \int_0^t\sup_{ r\in[0,s]}\E|z_r^n-x_r^{n}|^{p}\,\mathrm{d}s 
		\notag
		\\
		&\quad+C\E\int_{0}^{t}\big|\mu(\tau^n_s, x_{\eta^n_s}^n)-(\widehat\mu)_{\mathrm{tm}}^{n}\big(\tau^n_s,x_{\eta^n_s}^{n}\big)\big|^p\,\mathrm{d}s \notag
		\\
		&\leq Cn^{-p\min\{\beta+\frac{1}{2},1\}}  +C  \int_0^t\sup_{ r\in[0,s]}\E|z_r^n-x_r^{n}|^{p}\,\mathrm{d}s   
		\label{eq:Tmu}
	\end{align}
	for any $t\in[0,T]$.
	To estimate the second term $\mathscr{T}_\rho$ in \eqref{eq:error_y_n_x_n}, one applies  Hypotheses~\ref{asum:monotonocity}, \ref{asum:holder_time_diffusion_jump} along with  Lemma~\ref{lem:rho}, \ref{lem:estimate_first_term_MR} yielding the following for any $t \in [0, T]$
	\begin{align} \label{eq:Trho}
		\mathscr{T}_\rho: &= C\E\int_{0}^{t}|z_s^n-x_s^{n}|^{{p}}\,\mathrm{d}s 
		+C\E\int_{0}^{t}\big|\rho(s, x_s)-  \Gamma(\eta^n_s,s,x_{\eta^n_s}^{n})\big|^p\,\mathrm{d}s \notag
		\\
		&\leq
		C\E\int_{0}^{t}|z_s^n-x_s^{n}|^{{p}}\,\mathrm{d}s 
		+C\E\int_{0}^{t}\big|\rho(s, x_s)-\rho(\eta^n_s, x_s)\big|^p\,\mathrm{d}s	\notag
		\\
		&\qquad+C\E\int_{0}^{t}\big|\rho(\eta^n_s, x_s)-\rho(\eta^n_s, z_s^n)\big|^p\,\mathrm{d}s+C\E\int_{0}^{t}\big|\rho(\eta^n_s, z_s^n)-\rho(\eta^n_s, x_s^n)\big|^p\,\mathrm{d}s \notag
		\\
		&\qquad+
		C\E\int_{0}^{t}\big|\rho(\eta^n_s, x_s^n)-  \Gamma(\eta^n_s,s,x_{\eta^n_s}^{n})\big|^p\,\mathrm{d}s \notag
		\\
		&\leq Cn^{-p\min\{\beta+\frac{1}{2},1\}} +C\E\int_0^t|x_s-z_s^{n}|^{p}\,\mathrm{d}s +C\E\int_0^t|z_s^n-x_s^{n}|^{p}\,\mathrm{d}s  \notag
		\\
		&\leq Cn^{-p\min\{\beta+\frac{1}{2},1\}} +C \int_{0}^{t}\sup_{ r\in[0,s]}\E|z_r^n-x_r^{n}|^{p}\,\mathrm{d}s.
	\end{align}
	Substituting \eqref{eq:Tmu} and \eqref{eq:Trho} into \eqref{eq:error_y_n_x_n}, and applying {Gr\"onwall's} inequality, one concludes the proof.
	
	For the case \(  p< 2  \), one uses H\"older's inequality.
\end{proof}
\begin{proof}[Proof of Theorem \ref{thm:main_result}]
	For any \( t \in [0, T] \),  add and subtract the auxiliary process given in \eqref{eq:auxiliary_equation} to obtain
	\begin{align*}
		\E|x_t - x_t^{n}|^{p} \leq C\,\E|x_t - z_t^{n}|^{p} + C\,\E|z_t^{n} - x_t^{n}|^{p}.
	\end{align*}
	Applying Lemmas~\ref{lem:estimate_first_term_MR} and~\ref{lem:estimate_last_term_MR} to the terms on the right-hand side yields the desired result.	
\end{proof}

\section{Numerical Experiment} \label{sec:numerics}
In this section, the theoretical results are corroborated by numerical simulations for the stochastic FitzHugh--Nagumo system~\eqref{eq:sde-system}, which naturally falls within the scope of the analytical framework.

Over the time interval $t \in [0,1]$, the parameters of the SDE \eqref{eq:sde-system} are specified as follows
\[
I_{\mathrm{ext}}(t) = 25\bigl(1 - \sqrt{t}\bigr), \,\,
V_0 = 2, \,\,
R_0 = -1, \,\,
\sigma = 0.001, \,\,
\alpha = 0.8, \,\,
\gamma = 0.7, \,\,
\lambda = 0.8.
\]
Further,
recall the uniform time grid \( \{t_j\}_{j=0}^n \) on the interval \([0,1]\) with step size 
\(\delta t = 1 / n\), and define the  randomized evaluation time by
\[
\tau_j := t_j + \tau_j \,\delta t, 
\quad \tau_j \sim \mathcal{U}(0,1), 
\quad \text{for } j = 0,1,\ldots,n-1.
\]
Then, the \emph{randomized-tamed Milstein scheme}  for the SDE \eqref{eq:sde-system}  advances the solution as
\begin{equation} \label{eq:tamed-milstein}
	\begin{aligned}
		V_{t_{j+1}} &= V_{t_j} 
		+ \frac{1}{\mathcal{T}_j} \Big( V_{t_j} - \tfrac{V_{t_j}^3}{3}\Big) \delta t - \Big(R_{t_j} - I_{\mathrm{ext}}(\tau_j) \Big) \delta t 
		\\&\qquad\qquad+ \sigma\, V_{t_j} \,\delta w_j 
		+ \frac{1}{2} \sigma^2 V_{t_j} \left( (\delta w_j)^2 - \delta t \right), \\[0.6em]
		R_{t_{j+1}} &= R_{t_j} 
		+ \alpha \left( V_{t_j} + \gamma - \lambda R_{t_j} \right) \delta t,  
		\quad \text{for } j=0,1,\ldots,n-1
	\end{aligned}
\end{equation}
where  \,
\(
\delta w_j := w_{t_{j+1}} - w_{t_j} \,\, \mbox{ and }
\,\, 
\mathcal{T}_j := 1 + n^{-1} |V_{t_j}|^{2\xi}, 
\,\, 
\xi = 2.
\) Here, $\mathcal{T}_j$, defined in equation~\eqref{eq:taming_drift}, acts as a taming denominator in the nonlinear drift, controlling the super-linear growth of $V_{t_j}$.


\sloppy
	The reference solution is obtained using the randomized--tamed Milstein scheme \eqref{eq:tamed-milstein} 
	with a time step of \(2^{-18}\). 
	All results reported in the subsequent figure are based on 2000 independent simulation runs.

%
\begin{figure}[!htbp]
	\centering
	\includegraphics[width=0.55\textwidth,height=0.45\textwidth]{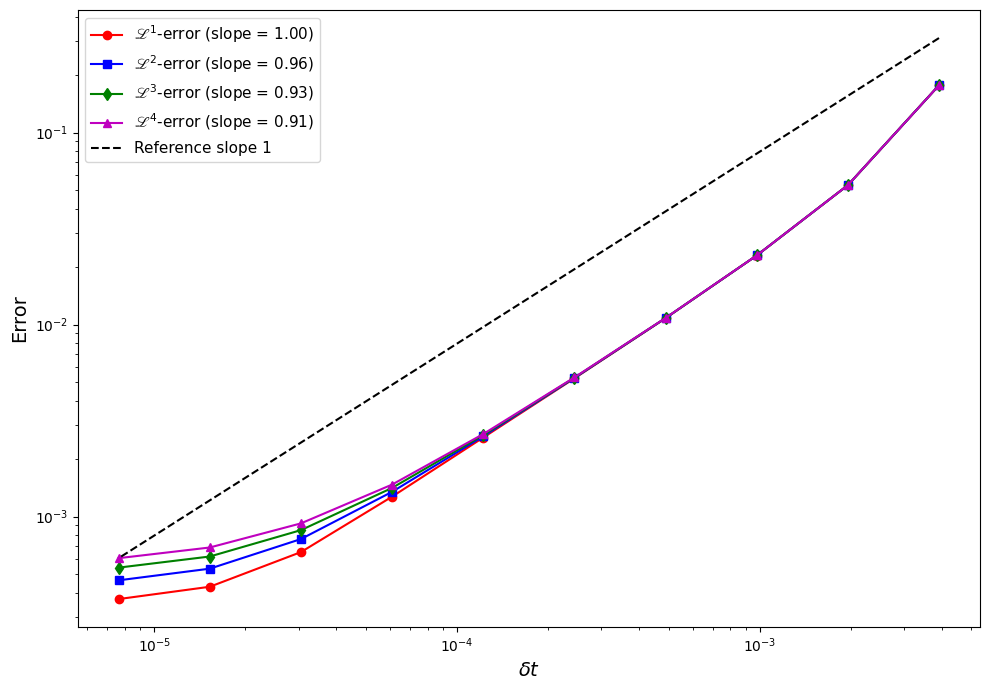}
	\caption{\textit{$\mathscr{L}^p$-error of the scheme \eqref{eq:tamed-milstein} for the SDE \eqref{eq:sde-system}.}}
\end{figure}

The numerical results exhibit convergence rates that are in close agreement with the theoretical prediction of order \(1.0\), thereby offering additional confirmation of Theorem~\ref{thm:main_result}.

\section*{Acknowledgements}
\sloppy
	This research was 
	funded by Centro de Modelamiento Matem\'atico (CMM) FB210005, and BASAL funds for centers of excellence from ANID-Chile. 


\end{document}